\documentclass[a4paper,11pt]{article}
\usepackage[utf8]{inputenc}
\usepackage[T1]{fontenc}

\usepackage{amsthm,amsmath}
\usepackage{tikz}
\usepackage{mathrsfs,amssymb,amsfonts} 
\usepackage{enumitem}
\usepackage{hyperref}
\usepackage[babel]{microtype}
\usepackage[english]{babel}
\usepackage[capitalise]{cleveref}

\usepackage{thmtools}
\usepackage{mathtools, comment}
\usepackage[nomath]{lmodern}
\usepackage{graphicx}
\usepackage{pgf,tikz,tkz-graph,subcaption}
\usetikzlibrary{arrows,shapes}
\usetikzlibrary{decorations.pathreplacing}
\usepackage{tkz-berge}
\hypersetup{colorlinks = true, linkcolor = blue, citecolor = blue, urlcolor = blue}

\newcommand*{\ceilfrac}[2]{\mathopen{}\left\lceil\frac{#1}{#2}\right\rceil\mathclose{}}
\newcommand*{\floorfrac}[2]{\mathopen{}\left\lfloor\frac{#1}{#2}\right\rfloor\mathclose{}}

\newcommand{\G}{\mathcal G}

\newcommand{\eps}{\varepsilon}

\allowdisplaybreaks

\usepackage[margin=1in]{geometry}
\newtheorem{defi}{Definition}

\newtheorem{thm}[defi]{Theorem}

\newcommand*{\myproofname}{Proof}

\newcommand*{\avp}{avp}

\newcommand{\diam}{diam}

\newcommand{\pn}{pn}

\title{The number and average length of subpaths in graphs}

\author{Stijn Cambie
 \thanks{Department of Computer Science, KU Leuven Campus Kulak-Kortrijk, 8500 Kortrijk, Belgium. Supported by a postdoctoral fellowship by the Research Foundation Flanders (FWO) with grant number 1225224N.} }%\and Nounours Cambie

\begin{document}
\maketitle
\begin{abstract}
    We study extremal questions on the average length and the number of subpaths in a graph.
    In particular, we prove the questions of Jamison (from $1983, 1984$) for the analogous concept of the average length of a subpath. Among other results, we prove that $K_n$ maximizes the average path length.
    % the latter being a conjecture which is still open for average subtree order.  
    % We also  the solution of two conjectures from Knor, Sedlar, \v{S}krekovski and Yang on the number of (sub)paths in a graph.
\end{abstract}

\section{Introduction}
The number of (sub)paths of length $k$ in a graph $G$ is denoted with $\pn_k(G)$. It is equal to the number of subgraphs of $G$ isomorphic to $P_{k+1}.$
Since paths are among the simplest graphs, they have been investigated in extremal questions.
E.g. the maximum of $\pn_3(G)$ has been determined in~\cite{GGPSTZ22} among planar graphs of a given order
and in~\cite{Byer01} for graphs of a fixed size.
Also Hamiltonian paths have been studied in different settings, see e.g.~\cite{AAR01,CFMNRZ21} for examples in other contexts. The number of Hamiltonian paths is counted by $\pn_{n-1}(G)$ (where $n$ is the order of the considered graph(s)).

The total number of paths in $G$, the (sub)path number of the graph, is denoted with $\pn(G)$, which equals $\sum_{i=0}^{n-1} \pn_i(G).$
We will call $\overline{\pn}(G)=\sum_{i=1}^{n-1} \pn_i(G)$ the reduced path number (not considering the single vertices as paths).
The average number of paths between two vertices is $\frac{ \overline{\pn}(G)}{\binom n2}.$
The latter has been studied in e.g.~\cite{vM01}.
The average number of paths gives an idea about the hardness of routing problems; the more potential paths there are, the harder it might be to find the shortest one. In this paper, we assume the considered graphs are connected.

Not only the number of paths in a graph has been studied. In various contexts, one has studied the number of induced subgraphs and number of subtrees.
See e.g.~\cite{Matula70,SW05,CJG23} (and cited/ citing papers) for a few examples. 

The quantity $\avp(G)=\frac{ \sum_{i=0}^{n-1} i \cdot \pn_i(G) }{\pn(G)}$, is the average length of the subpaths of a graph, which is analogous to the mean subtree order (average of the order over all subtrees, as defined in~\cite{jamison_average_1983}) of a graph, but focused on subpaths.
We also defined the reduced version $\overline{\avp}(G)=\frac{ \sum_{i=1}^{n-1} i \cdot \pn_i(G) }{\overline{\pn}(G)}$.

Note that the latter can be useful when studying potential shortest paths (random walks~\cite{Lovasz96} which are not allowed to be self-intersecting).

For a tree $T$, the average path length is related to the average distance (also called average shortest path length) $\mu.$
The total distance (Wiener index) of a graph, $W(G)$, equals the sum of the distance (lengths of a shortest path) between any two vertices, and $\mu(G)=\frac{W(G)}{\binom n2}.$
Now $\overline{\avp}(T)=\mu(T)$ and
$\avp(T)=\frac{W(T)}{\binom{n+1}{2}}= \frac{n-1}{n+1}\mu(T).$

Being a natural and basic concept, it is good to have some intuition on the extremal behaviour and some characteristics under graph operations.
We answer the questions from Jamison~\cite{jamison_average_1983, jamison_monotonicity_1984} and some further problems that arose from his initial work, when focusing on subpaths instead of subtrees.
See~\cref{thm:main}.

We remark that the analogue of the caterpillar conjecture~\cite[Prob.~(7.1)]{jamison_average_1983} was trivially resolved, since every tree $T$ of order $n$ satisfies $ \avp(T) \le \avp(P_n)$. 
% Here one can observe that $P_n$ minimizes the average subtree order, showing a stark contrast.
Also the analogue of determining the maximum among graphs is resolved, $ \avp(G) \le \avp(K_n)$,
while it is open for the average subtree order~\cite{CJW25}.

In~\cite{KSSY25}, the authors posed a few natural questions on the number of subpaths in a graph. They wondered about the extrema of the path number of (connected) regular graphs. For the first non-trivial case, they conjectured~\cite[Conj.~16]{KSSY25} that every cubic graph satisfies $\pn(G)=\omega\left( \sqrt 2 ^n\right)$ by conjecturing an explicit extremal graph in this class.
We prove that this is far from the truth and the correct bound is quadratic in the order.
% While the number of connected vertex subsets is exponential in the order (see e.g.~\cite{CJG23}), the number of paths can be polynomially large.
Finally, we prove that among triangle-free graphs of order $n$, $K_{ \floorfrac n2, \ceilfrac n2}$ (uniquely) maximizes the number of subpaths, confirming~\cite[Conj.~18]{KSSY25}.

\textbf{Reminder of some definitions and notation}

The minimum degree of a graph is denoted by $\delta(G).$
The local average path length of a graph $G$ in a vertex $v$, $\avp(G,v)$ is the average length of all paths containing $v$.
When not considering a single vertex as a subpath of a graph parameter $p$, we speak about the reduced version, denoted by $\overline p.$
A caterpillar is a tree which becomes a path after removing its leaves.

\subsection{Note on future work}

This paper addresses the analogs of the problems of~\cite{jamison_average_1983,jamison_monotonicity_1984} in the setting of the average length of subpaths.
Some of those are/were open at the moment of writing.
The final resolutions should appear in the near future. In particular, an updated version of~\cite{CJW25} will contain both directions of the initial version of item 8 in~\cref{thm:main}.

\section{Some results on the average length of subpaths}\label{sec:avp}

\begin{thm}\label{thm:main}
    \begin{enumerate}
    \item \label{itm:7.1&7.2} For every tree $T$ of order $n$, $\avp(S_n) \le \avp(T) \le \avp(P_n)$.
    \item\label{itm:7.3} There are non-isomorphic trees with the same order and average path length.
    \item\label{itm:7.4} The ratio of the average length of subpaths in $G$ containing a vertex $v$ ($\avp(G,v)$), and $\avp(G)$ is dense in $\mathbb R^+.$ Restricted to trees, it is dense in $(1/2, +\infty).$ 
    \item\label{itm:7.5} The vertex $v$ in which $\avp(G,v)$ is maximum can have any degree.
    \item\label{itm:7.6} Every tree which is not a star has a $1$-associate with lower average path length.
    \item\label{itm:5.6} Both removing an edge in a graph and contracting an edge in a tree can increase the $\avp$.
    \item\label{itm:Vinceconj7} For every $r \ge 3$, there are graphs $G$ with $\delta(G) \ge r$ for which $\avp(G)= O(log(n)).$
    \item\label{itm:2} For every graph $G$ of order $n$, $\avp(S_n) \le \avp(G) \le \avp(K_n)$.
    \end{enumerate}
\end{thm}

The different items of the theorem immediately address the analogues of corresponding problems;~\cref{itm:7.1&7.2} for \cite[Prob.~(7.1) \& (7.2)]{jamison_average_1983}, \cref{itm:7.3} \cite[Prob.~(7.3)]{jamison_average_1983}, \cref{itm:7.4} \cite[Prob.~(7.4)]{jamison_average_1983}, \cref{itm:7.5} \cite[Prob.~(7.5)]{jamison_average_1983}, \cref{itm:7.6} \cite[Prob.~(7.6)]{jamison_average_1983}, \cref{itm:5.6} tackes both \cite[Conj.~5.6]{jamison_monotonicity_1984} and~\cite[Conj.~7.4]{CGMV18}, and
\cref{itm:Vinceconj7} for~\cite[Conj. 7]{Vince20}.
While~\cref{itm:5.6} disproves~\cite[Conj.~5]{CGMV18}, the intended corollary that $K_n$ would be extremal is proven in our context as~\cref{itm:2}.

\begin{proof}
\begin{enumerate}
    \item % \Cref{itm:7.1&7.2} 
    This follows directly from the folklore result $\mu(S_n) \le \mu(T) \le \mu(P_n)$.
    \item This is immediate from the pigeon hole principle, since there are exponentially many trees of order $n$ and the total distance is polynomially bounded by $n.$ A stronger version is proven in~\cite[cor.~1]{WW14}. 
    \item For a tree, $\avp(T, v) > \frac 12 \avp(T)$ as a corollary to a general bound between the transmission  $\sigma$ (sum of distances towards a fixed vertex) and the total distance, $W$,~\cite[Thm.~2.1]{HK14} and~\cite[Lem.~26]{CD25}, $(n-1)\sigma(v) \ge W(G)$ for all $v \in V(G)$, and observing that the average length of paths containing $v$ as an internal vertex is larger than the average length of paths having $v$ as an end vertex, leading to
    \[
 \frac{\avp(T,v)}{\avp(T)}
 \geq
 \frac{\sigma_T(v)/n}{W(T)/\binom{n+1}{2}}
 \geq \frac{n+1}{2(n-1)}
 >\frac12.
\]

    In the other direction, take a broom $B_{x,y}$ of order $n$ (concatenation of path $P_{y}$ and star $S_{x}$), where $y/n \to b$ and $x/n \to (1-b).$
    
    If $z$ is a leaf of the star, standard computations verify that
\[
 \avp(B_{x,y})
 =\frac{W(B_{x,y})}{\binom{n+1}{2}}
 \sim \frac{b^2(3-2b)}3n \mbox{ and }
\]
\[
 \avp(B_{x,y},z)=\frac{\sigma(z)}n
 \sim\frac{b^2}{2}n, \mbox{ and thus}
\]
\[
 \frac{\avp(B_{x,y},z)}{\avp(B_{x,y})}
 \longrightarrow
 f(b):=\frac{3}{2(3-2b)}.
\]
As $b$ varies in $(0,1)$, the values $f(b)$ fill the interval
$(1/2,3/2)$ in the limit.

If $w$ is the only non-leaf neighbour of the high degree vertex in the broom, then
\[
 \avp(B_{x,y},w)\sim\frac b2 n
\]
and hence
\[
 \frac{\avp(B_{x,y},w)}{\avp(B_{x,y})}
 \longrightarrow
 g(b):=\frac{3}{2b(3-2b)}.
\]
The function $g$ has minimum $4/3$, attained at $b=3/4$, and tends
to infinity as $b\downarrow0$.

    Now we turn to the graph setting. Connect the vertex $u$ of a clique $K_a$ and a vertex $w$ of $K_b$ with a long path, and add a pendent edge $uv.$ Here $a,b=o(n)$ and the number of paths, $N_u$ and $N_w$, in $K_a$ and $K_b$ having $u$ resp. $w$ as an end vertex, satisfy $N_u \gg n N_w\gg n^3.$
    Then $\avp(G) \sim n$ and $\avp(G,v)=o(n).$ Again varying the ratios, and the neighbour of $v$ on the central path, will result in density results.

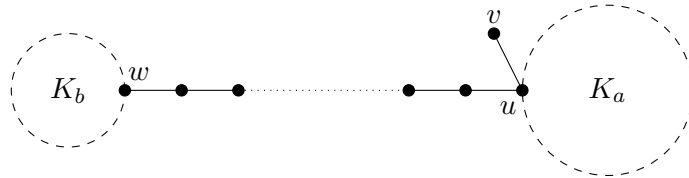
\begin{figure}[ht]
    \centering
    \begin{tikzpicture}[scale=0.75]

\draw[dashed] (0,0) ellipse (1.cm and 1cm);
\draw[dashed] (9.5,0) ellipse (1.5cm and 1.5cm);
\draw(1,0)--(3,0);
\draw[dotted] (3,0)--(6,0);
\draw(6,0)--(8,0);
\foreach \x in {1,2,3,6,7,8}{
\draw[fill] (\x,0) circle (0.1);
}
\draw[fill] (7.5,1) circle (0.1);
\draw(7.5,1)--(8,0);

\node at (1.25,.3) {$w$};
\node at (7.5,1.3) {$v$};
\node at (7.75,-0.3) {$u$};
\node at (9.5,0) {$K_a$};
\node at (0,0) {$K_b$};
    \end{tikzpicture}
    \caption{A graph with $\avp(G)=n-o(n)$ and $\avp(G,v)=o(n)$ }
    \label{fig:enter-label}
\end{figure}

    \item A broom shows that degree $1$ is possible.
    For degree $b+2$ at least $2$, take a path $P_3$, connect each of its endvertices with $a$ additional vertices and the central vertex with $b$ additional vertices, to obtain a caterpillar of order $2a+b+3.$
    If $a \gg b$, the central vertex is the one maximizing the local average path length.

    \begin{figure}[ht]
    \centering
    \begin{tikzpicture}

\draw(0,0)--(2,0);

\foreach \x in {-1,-0.7,0.7,1}{
\draw[fill] (-1,\x) circle (0.1);
\draw[fill] (3,\x) circle (0.1);
\draw(0,0)--(-1,\x);
\draw(3,\x)--(2,0);
}
\foreach \x in {0,1,2}{
\draw[fill] (\x,0) circle (0.1);
}

\draw[dotted] (-1,0.5)--(-1,-0.5);
\draw[dotted] (3,0.5)--(3,-0.5);
\draw[dotted] (0.75,1)--(1.75,1);

\foreach \x in {-0.7,-0.4,0.4,0.7}{
\draw[fill] (1+\x,1) circle (0.1);
\draw(1,0)--(1+\x,1);
}

\node at (1,-0.25) {$v$};
\node at (-1.25,0) {$a$};
\node at (3.25,0) {$a$};
\node at (1,1.25) {$b$};
    \end{tikzpicture}
    \caption{A graph where the maximum local average path length $\avp(G,v)$ appears in a vertex of degree $b+2.$}
\end{figure}
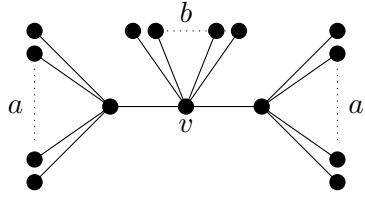

    \item Take a diameter $v_0v_1 \ldots v_d$ (a path between $v_0$ and $v_d$ where $d(v_0,v_d)=\diam(T)$), and assume without loss of generality that $\deg(v_1) \le \deg(v_{d-1}).$
    Removing the edge $v_0v_1$ and adding $v_0v_2$ results in a tree (a $1$-associate) with lower total distance.
    \item Take a path $P_{n-1}$ (for $n$ sufficiently large) and add a pendent vertex $v$ to one of its central vertices. Then contracting this pendent edge increases the average path length. The latter being true since $\avp(G,v) \sim \frac n4$ while $\avp(P_{n-1}) \sim \overline{\avp}(P_{n-1})=\frac n3.$

    Connect the end vertices of a $K_2$ with a fixed central vertex of a path $P_{n-2}$. Deleting the initial $K_2$ will increase $\avp$ for $n$ sufficiently large, analogously to the contraction case. 
    (By connecting the two vertices which are at distance $\eps n$ from the leaves of a path $P_n$, one can note that the average path length can also increase linearly with $n$ by adding one edge)
    
    \item Taking a $r$-ary tree and connecting its leafs to $K_r$s will do the job.
    One can even adapt this to $r$-regular constructions (see~\cref{fig:pn_quadratic_regular} for an idea when $r=3$). Here the construction for $r$ odd is easy. For $r$ even, modified regular constructions work as well (e.g. starting from a particular cactus graph instead of tree).
    
    \item The lower bound is easy.
    First, for a tree $T$ we have $\pn_0(T)=n,$ $\pn_1(T) = n-1=\pn_1(S_n)$ and $\sum_{i \ge 2} \pn_i(T) = \binom{n-1}{2}=\pn_2(S_n)$. Hence $\avp(T) \ge \avp(S_n),$ and equality only occurs when $T$ has diameter $2$, i.e., $T=S_n.$
    Starting from a spanning tree of $G$, every additional edge implies that at least $n$ paths are added (extend the edge with a path towards every other vertex).
    Thus $\pn_1$ raises by $1$, and $\sum_{i \ge 2} \pn_i$ by at least $n-1$.
    The average length of these additional paths is larger than $\avp(S_n)$, and thus the result follows inductively.
    
    For the upper bound, one can observe that every graph $G$ of order $n$ and integer $k \ge 1$ satisfies $\frac{\pn_{k}(K_n)}{\pn_{k-1}(K_n)}\ge \frac{\pn_{k}(G)}{\pn_{k-1}(G)}$.
    The latter since every path of length $k-1$ with a particular start vertex $v$ can be extended to such a path of length $k$ in at most $n-k$ ways (and equality is attained by $K_n$).
    From this, one can (e.g. by proving the inequality inductively on the average of path lengths bounded by a certain $k$) derive that $\avp(G) \le \avp(K_n)$. The latter is analogous to~\cite[Sec.~5]{CJW25}, where the implications are written down in more detail.
    \qedhere
    \end{enumerate}
\end{proof}

\section{Results on the number of subpaths}\label{sec:pn}

The following theorem disproves~\cite[Conj.~16]{KSSY25} in a strong way, proving that a quadratic instead of exponential behaviour.

\begin{thm}\label{thm:conj16ontkracht}
    Let $r \ge 3.$
    Among $r$-regular graphs of order $n$, the minimum path number is quadratic in its order.
\end{thm}

\begin{proof}
    Take $\floorfrac{n}{r+2}$ (small) graphs with total order $n$, such that each such graph has the property that all of its vertices except one has degree $r$, and one has degree $r-2.$
    These graphs will be the (building) blocks of our graph.
    We connect all vertices of degree $r-2$ with an additional cycle, to obtain the graph $G$.
    The number of paths in such a small graph is bounded by a constant $C_r$ (each such graph has order bounded by $2r+3$).
    The number of paths between any two vertices in $G$ is bounded by $2C_r^2,$ since such a path is composed by a path between two initial small graphs and a path in these small graphs (possibly it is the same graph or a single vertex). 
    Hence $\pn(G) \le 2C_r^2 \binom{n}{2}+n.$ Here the additional $n$ is for the single vertices which are $P_1$s.
    When $r$ is odd, one can modify the construction with a tree instead of a cycle, and obtain $\pn(G) \le C_r^2 \binom{n}{2}+n$, as there is now only one instead of two subpaths between two blocks.
    Two examples are depicted in~\cref{fig:pn_quadratic_regular}.

    Due to $\binom{n}{2}+n$ being a lower bound, the bound is best possible up to a constant.
\end{proof}

It may well be that with a careful analysis, the sharp bounds can be proved and the extremal graphs are not unique at all for large values of $n$. E.g., by attaching optimal subcubic graphs to the leaves of any tree of the correct size when $r=3.$

\begin{figure}[h!]
 \centering
 \begin{tikzpicture}[scale=1.2]
% \foreach \x in {0,120,240}{
%   \draw[fill] (\x:2) circle (0.15);
%   \draw[fill] (\x+15:2.5) circle (0.15);
%    \draw[fill] (\x-15:2.5) circle (0.15);
%    \draw[fill] (\x+15:3) circle (0.15);
%    \draw[fill] (\x-15:3) circle (0.15);
% }
% \foreach \x in {120,240}{

% }
%    \draw[fill] (15:3.5) circle (0.15);
%    \draw[fill] (-15:3.5) circle (0.15);

\draw[fill] (0,0) circle (0.15);
\draw[fill] (1,0) circle (0.15);
\draw[fill] (0,1) circle (0.15);
\draw[fill] (0,-1) circle (0.15);
\foreach \x in {2,3,4}{
\draw[fill] (\x,0.5) circle (0.15);
\draw[fill] (\x,-0.5) circle (0.15);
}

\draw (0,1)--(0,-1);
\draw (0,0)--(1,0);

\draw (-0.5,2)--(0,1)--(0.5,2);
\draw (-0.5,2)--(-0.5,3)--(0.5,2);
\draw (-0.5,2)--(0.5,3)--(0.5,2);
\draw (0.5,3)--(-0.5,3);

\draw (-0.5,-2)--(0,-1)--(0.5,-2);
\draw (-0.5,-2)--(-0.5,-3)--(0.5,-2);
\draw (-0.5,-2)--(0.5,-3)--(0.5,-2);
\draw (0.5,-3)--(-0.5,-3);

\draw (1,0)--(2,0.5)--(4,0.5)--(4,-0.5)--(2,-0.5)--(1,0);
\draw (2,0.5)--(2,-0.5);
\draw (3,0.5)--(4,-0.5);
\draw (4,0.5)--(3,-0.5);

\foreach \x/\y in {0.5/2,0.5/-2,-0.5/2,-0.5/-2}{
\draw[fill] (\x,\y) circle (0.15);
}
\foreach \x/\y in {0.5/3,0.5/-3,-0.5/3,-0.5/-3}{
\draw[fill] (\x,\y) circle (0.15);
}
% \draw[fill] (0.5,2) circle (0.15);
% \draw[fill] (0,-1) circle (0.15);

\end{tikzpicture}\qquad
 \begin{tikzpicture}
  \foreach \x in {0,40,...,320}{
  \draw[fill] (\x:2.5) circle (0.15);
  \draw[fill] (\x:4) circle (0.15);
  \draw[fill] (\x+10:3) circle (0.15);
   \draw[fill] (\x-10:3) circle (0.15);
   \draw[fill] (\x+10:3.5) circle (0.15);
   \draw[fill] (\x-10:3.5) circle (0.15);
   \draw (\x:4)--(\x-10:3.5)--(\x+10:3.5)--(\x:4);
   \draw (\x-10:3)--(\x:4)--(\x+10:3)--(\x-10:3.5)--(\x-10:3)--(\x+10:3.5)--(\x+10:3);
   \draw (\x-10:3)--(\x:2.5)--(\x+10:3);
   \draw (\x-40:2.5)--(\x:2.5);
}
 % \foreach \x in {45,135,225,315}{
 % \draw[fill] (\x:3.26640741219) circle (0.15);
 % \draw(\x+90:3.26640741219) -- (\x:3.26640741219);
 % }

\end{tikzpicture}

\caption{Example of constructed cubic and quartic graphs with low average path length and $\Theta(n^2)$ many paths }
\label{fig:pn_quadratic_regular}
\end{figure}
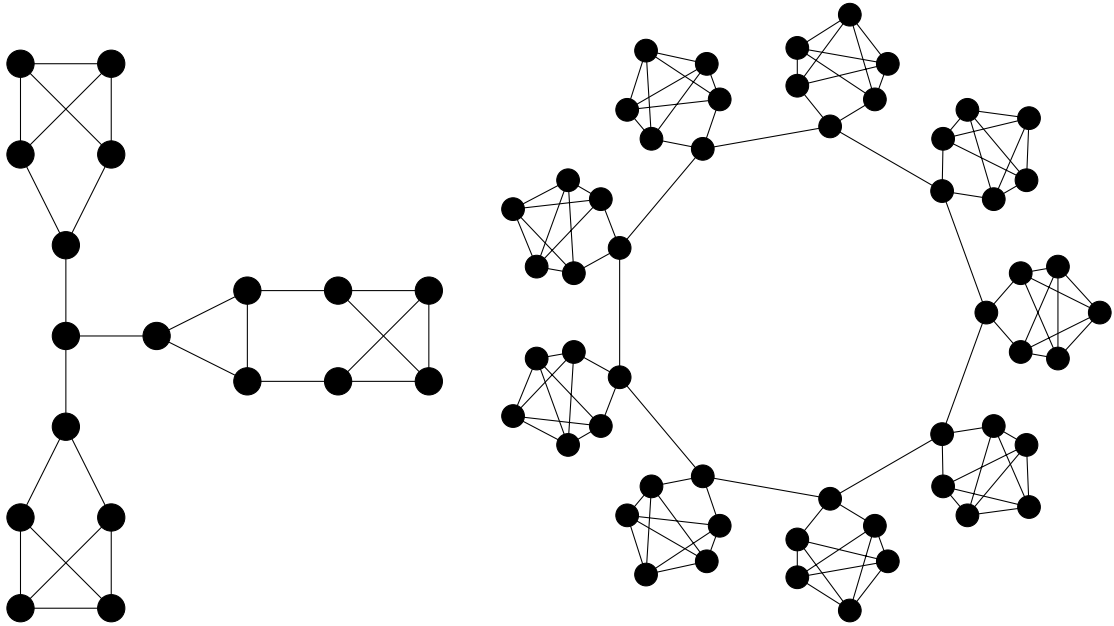

The maximum among $r$-regular graphs will be exponential, and so similar as in~\cite{CJG23} (and cited papers therein) one may wonder about the value $\limsup_{G \in \G_r} \sqrt[n]{\pn(G)}$ where $\G_r$ is the family of $r$-regular graphs.
This and the characterization of the extremal graphs in those classes, seem a very hard problem. In particular, the author doubts about the possibility for 
 a general exact determination of~\cite[Prob.~17]{KSSY25}.
For $n \le 16,$ the cubic graph maximizing $\pn(G)$ corresponds with the graphs maximizing the number of connected induced subgraphs determined in~\cite{CJG23}. 
Nevertheless, for $n=18$, the construction is different since the Pappus graph (graph6string $\texttt{Qs??OGC@?O@?GgCcCG\_aOOD@?g?}$) has more subpaths. Despite the difference, it may be interesting to consider if Moore graphs are also extremal for the path number.

Finally, we prove~\cite[Conj.~18]{KSSY25}, the unique extremal graph among triangle-free graphs of a given order. Notice that Gerbner~\cite{Gerbner23} proved a related asymptotic version for $\pn_i$ for $K_r$-free graphs.
\begin{thm}
    Among triangle-free graphs, the balanced complete bipartite graph maximizes (uniquely) the number of subpaths.
\end{thm}

\begin{proof}
Write $T_n=K_{\lfloor n/2\rfloor,\lceil n/2\rceil}$ (it is a Turan graph)
and let $m_s(G)$ denote the number of matchings of size $s$ in $G$.

We first record that, among triangle-free graphs of order $n$ and for every $s\geq0$,
\begin{equation}\label{eq:matching-bound}
        m_s(G)\le m_s(T_n).
\end{equation}

We prove~\eqref{eq:matching-bound} by induction on $s$. The case $s=0$ is trivial. Every
matching of size $s$ is counted once for each of its $s$ edges, and
therefore
\[
 s\,m_s(G)
   =\sum_{uv\in E(G)}m_{s-1}(G-\{u,v\}).
\]
Since $G-\{u,v\}$ is triangle-free, induction and Mantel's theorem give
\[
 s \cdot m_s(G)
 \leq e(G)m_{s-1}(T_{n-2})
 \leq
 \left\lfloor\frac{n^2}{4}\right\rfloor
 m_{s-1}(T_{n-2})= s \cdot m_s(T_n).
\]

We now bound the paths using matchings. First consider a path of odd
length $2s-1$. Orient the path and take the edges in positions
\[
        1,3,\ldots,2s-1.
\]
These edges form a matching of size $s$, together with an ordering of
its edges.

Fix an ordered matching $e_1,\ldots,e_s$. There are at most two ways
to orient its edges so that they form a path in the prescribed order.
Indeed, there are two choices for the orientation of $e_1$. Once the
terminal vertex of $e_i$ has been chosen, it is adjacent to at most
one endpoint of $e_{i+1}$, since adjacency to both endpoints would
create a triangle with $e_{i+1}$.

Consequently, the number of oriented paths of length $2s-1$ is at most $2s!\,m_s(G).$
Every unoriented path has two orientations, so
\begin{equation}\label{eq:odd-path}
        \pn_{2s-1}(G)\leq s!\,m_s(G).
\end{equation}

For paths of even length $2s$, use the same matching formed by the
first $2s$ vertices and then append the last vertex. For a fixed
ordered matching, there are again at most two compatible orientations.
If both exist, their terminal vertices are the two endpoints $x,y$ of
the last matching edge. Since $xy\in E(G)$ and $G$ is triangle-free,
\[
        N(x)\cap N(y)=\emptyset.
\]
Hence the total number of unused vertices that can extend these two
orientations is at most $n-2s$. The same bound is immediate if only
one orientation exists. Therefore
\begin{equation}\label{eq:even-path}
        \pn_{2s}(G)
        \leq \frac{s!}{2}(n-2s)m_s(G).
\end{equation}

Both \eqref{eq:odd-path} and \eqref{eq:even-path} are equalities for
every complete bipartite graph: for each ordered matching there are
exactly two compatible orientations, and in the even case the numbers
of possible final vertices add up to $n-2s$. Combining these
equalities with \eqref{eq:matching-bound}, we obtain
\[
        \pn_\ell(G)\leq\pn_\ell(T_n)
\]
for every $0\leq\ell\leq n-1$. Summing over $\ell$ gives
\[
        \pn(G)\leq\pn(T_n).
\]

Finally, equality in the sum forces equality for paths of length one.
Thus
\[
        e(G)=\pn_1(G)=\pn_1(T_n)
             =\left\lfloor\frac{n^2}{4}\right\rfloor.
\]
The equality case of Mantel's theorem implies $G\cong T_n$. Hence the
balanced complete bipartite graph is the unique maximizer.
\end{proof}

% \bibliographystyle{abbrv}
% \bibliography{ref}

\end{document}